\documentclass[11 pt]{amsart}
\usepackage{amssymb}
\usepackage{amsmath}
\usepackage{amsthm}
\usepackage{bbm}
\usepackage{dsfont}
\newtheorem{thm}{Theorem}
\newtheorem{prop}{Proposition}

\newtheorem{lem}{Lemma}

\newtheorem{cor}{Corollary}
\newtheorem{rem}{Remark}

\numberwithin{equation}{section}

\DeclareMathOperator{\oc}{\xrightarrow[]{o}}

\begin{document}
	
\title{AMS Journal Sample}
	
\author{Y. A. Dabboorasad$^{1}$, E. Y. Emelyanov$^1$, M. A. A. Marabeh$^1$}

\address{$^{1}$ Department of Mathematics, Middle East Technical University, Ankara, 06800 Turkey.} 
\email{yousef.dabboorasad@metu.edu.tr, yasad@iugaza.edu.ps, ysf\_atef@hotmail.com, eduard@metu.edu.tr, mohammad.marabeh@metu.edu.tr, m.maraabeh@gmail.com}

\subjclass[2010]{}
	
\date{28.05.2017}
	
	
\title{Order convergence in infinite-dimensional vector lattices is not topological}
	
\begin{abstract}
In this note, we show that the order convergence in a vector lattice $X$ is not topological unless $\dim X<\infty$. 
Furthermore, we show that, in atomic order continuous Banach lattices, the order convergence is topological on order intervals.
\end{abstract}

\maketitle

\section{Introduction}

A net $(x_\alpha)_{\alpha \in A}$ in a vector lattice $X$ is \textit{order convergent} to a vector $x \in X$ 
if there exists a net $(y_\beta)_{\beta \in B}$ in $X$ such that $y_\beta \downarrow 0$ and, for each $\beta \in B$, there is an
$\alpha_\beta \in A$ satisfying $\lvert x_\alpha - x \rvert \leq y_\beta$ for all $\alpha \geq \alpha_\beta$. 
In this case, we write $x_\alpha \xrightarrow{o} x$. It should be clear that an order convergent net has an order bounded tail. 
A net $x_\alpha$ in $X$ is said to be \textit{unbounded order convergent} to a vector $x$ if, for any 
$u\in X_{+}$, $\lvert x_{\alpha}-x\rvert\wedge u\xrightarrow{o}0$. 
In this case, we say that the net $x_{\alpha}$ $uo$-converges to $x$, and write $x_{\alpha}\xrightarrow{uo}x$. 
Clearly, order convergence implies $uo$-convergence, and they coincide for order bounded nets. 
For a measure space $(\Omega,\Sigma,\mu)$ and a sequence $f_n$ in $L_p(\mu)$ ($0\le p\leq\infty$), we have $f_n\xrightarrow{uo}0$ 
iff $f_n\to 0$ almost everywhere; see, e.g., \cite[Remark 3.4]{GTX}. Hence, $f_n\xrightarrow{o}0$ in $L_p(\mu)$ iff $f_n\to 0$ almost 
everywhere and $f_n$ is order bounded in $L_p(\mu)$. It is known that almost everywhere convergence is not topological in general, i.e. 
there may not be a topology such that the convergence with respect to this topology is the same as $a.e.$-convergence; 
see for example \cite{O66}. Thus, the unbounded order convergence is not topological in general.

A net $x_\alpha$ in a normed lattice $X$ is \textit{unbounded norm convergent} (\textit{$un$-convergent}) to a vector $x$ if, 
for all $u\in X_{+}$, $\big\lVert\lvert x_{\alpha}-x\rvert\wedge u\big\rVert\to 0$ (cf. \cite{KT,T,DOT,KMT}). 
In this case, we write $x_{\alpha}\xrightarrow{un}x$. Clearly, norm convergence implies $un$-convergence, and they agree for order bounded nets. 
Unlike order and unbounded order convergences, $un$-convergence is always topological, and the corresponding topology is referred to as the 
\textit{un-topology} (see \cite[Section 7]{DOT}). The $un$-topology has been recently investigated in detail in \cite{KMT}. 

Recall that a net $x_\alpha$ in a vector lattice $X$ is \textit{relatively uniformly convergent} to a vector $x$ if there is $u\in X_+$ such that for any $\varepsilon >0$ there is $\alpha_\varepsilon$ satisfying $\lvert x_\alpha - x \rvert \leq \varepsilon u$ for all $\alpha \geq \alpha_\varepsilon$. In this case, we write $x_\alpha \xrightarrow{ru}x$. In Archimedean vector lattices relatively uniformly convergence implies order convergence. 

An element $a>0$ in a vector lattice $X$ is called an \emph{atom} whenever, for every $x\in[0,a]$, there is a real $\lambda\geq 0$ such that $x=\lambda a$. 
It is known that the band $B_a$ generated by an atom $a$ is a projection band and $B_a=span\{a\}$. 
A vector lattice $X$ is called \emph{atomic} if the band generated by its atoms is $X$. For any atom $a$, let $P_a$ be the band projection corresponding to $B_a$. 

A normed lattice $(X,\lVert\cdot\rVert)$ is said to be \textit{order continuous} if, for every net $x_\alpha$ in $X$ with $x_\alpha\downarrow 0$, it holds  
$\lVert x_\alpha\rVert\downarrow 0$ (or, equivalently, $x_\alpha\oc 0$ in $X$ implies $\lVert x_\alpha\rVert\to 0$). 
Clearly, in order continuous normed lattices, $uo$-convergence implies $un$-convergence.

Since the order convergence could be easily not topological, many researchers investigated classes of ordered topological spaces, in which order convergence of nets (or sequences) 
agrees with the topological convergence. For instance, in \cite{DeMarr64/1}, DeMarr proved that a locally convex space $(X,\tau)$ can be made into an ordered vector space 
such that the convergence of nets with respect to $\tau$ is equivalent to order convergence if and only if $X$ is normable. In \cite[Theorem 1]{DeMarr64/2}, DeMarr showed 
that any locally convex space $(X,\tau)$ can be embedded into an appropriate ordered vector space $E$ such that $x_\alpha \xrightarrow{\tau}0$ iff $x_\alpha\xrightarrow{uo}0$ 
in $E$ for any net $x_\alpha$ in $X$.

In \cite[Theorem 1]{VK73}, the authors characterized ordered normed spaces in which the order convergence of nets coincides with the norm convergence. Also, they characterized 
ordered normed spaces in which order convergence and norm convergence coincide for sequences; see \cite[Theorem 3]{VK73}. 

As an extension of the work in \cite{DeMarr64/1,DeMarr64/2,VK73}, Chuchaev investigated ordered locally convex spaces, where the topological convergence agrees 
with the order convergence of eventually topologically bounded nets; see, for example, Theorem 2.3, Propositions 2.4, 2.5, and 2.6 in \cite{Ch76}. 
In addition, he studied ordered locally convex spaces, where the topological convergence is equivalent to the order convergence of sequences; see, for example,
Propositions 3.2, 3.3, 3.4 and 3.5 in \cite{Ch76}.

In this paper, our main result is that if $(X,\tau)$ is a topological vector lattice such that $x_\alpha\xrightarrow{\tau}0$ iff $x_\alpha\xrightarrow{o}0$ 
for any net $x_\alpha$ in $X$, then $\dim X<\infty$; see Theorem \ref{main}. It should be noticed that Theorem \ref{main} was proven in the case of ordered 
normed spaces with minihedral cones (see \cite[Theorem 2]{VK73}), in the case of Banach lattices (see \cite{W74}), and in the case of normed lattices 
(see \cite[Theorem 2]{Goro96}). A useful characterization of $uo$-convergence in atomic vector lattices is given in Proposition \ref{atomipointwiseconvergence}. 
In addition, we show that the order convergence is topological on order intervals of atomic order continuous Banach lattices; see Corollary \ref{o-convergence is topological}. 
A partial converse of Corollary \ref{o-convergence is topological} is given in Theorem \ref{converse}.

Throughout this paper, all vector lattices are assumed to be Archimedean.

\section{Order convergence is not topological}\label{S1}

In this section, we show that order convergence is topological only in finite-dimensional vector lattices. We begin with the following technical lemma.

\begin{lem}\label{lemma 1}
Let $(X,\tau)$ be a topological vector lattice in which $x_\alpha \xrightarrow{\tau}0$ implies $x_\alpha \xrightarrow{o}0$ for any net $x_\alpha$. 
The following statements hold.
\begin{itemize}
\item[(i)] There is a strong unit $e\in X$.
\item[(ii)] For any net $(x_\alpha)$ in $X$, if $x_\alpha\xrightarrow{\tau}0$ then $x_\alpha\xrightarrow{\lVert\cdot\rVert_e}0$, 
where $\lVert x\rVert_e:=\inf\{\lambda>0:\lvert x\rvert\leq\lambda e\}$.
\end{itemize}

\end{lem}

\begin{proof}
Let $\mathcal{N}$ be the zero neighborhood base of $\tau$.
\begin{itemize}
\item[(i)]  
Let $\Delta:=\{(y,U): U\in\mathcal{N}\ \text{and}\ y\in U\}$ be 
ordered by $(y_1,U_1)\leq(y_2,U_2)$ iff $U_1 \supseteq U_2$. 
Under this order, $\Delta$ is directed upward. For each $\alpha\in\Delta$, 
let $x_\alpha=x_{(y,U)}:=y$. Clearly, $x_\alpha\xrightarrow{\tau}0$. 
Now the assumption assures that $x_\alpha\xrightarrow{o}0$. 
So there are $(y_0,U_0)\in\Delta$ and $e\in X_+$ such that, for all $(y,U)\geq(y_0,U_0)$, we have $x_{(y,U)}=y\in[-e,e]$. 
In particular, $U_0\subseteq[-e,e]$. Since $U_0$ is absorbing, for every $x\in X$, there is $n\in\mathbb{N}$ satisfying $\lvert x\rvert\in nU_0$
and so $\lvert x\rvert\leq ne$. Hence, $e$ is a strong unit.

\item[(ii)] Since $e$ is a strong unit then $I_e=X$, where $I_e$ is the ideal generated by $e$. 
For each $x\in X$, let $\lVert x\rVert_e:=\inf\{\lambda>0:\lvert x\rvert\leq\lambda e\}$. Then $(X,\lVert\cdot\rVert_e)$ is a normed lattice; see, for example, \cite[Theorem 2.55]{AT07}. 
Suppose $x_\alpha\xrightarrow{\tau}0$. Let $U_0$ be the zero neighborhood as in part (i). Given $\varepsilon>0$. Then $\varepsilon U_0$ is also a zero neighborhood. 
Hence, there is $\alpha_\varepsilon$ such that $x_\alpha\in\varepsilon U_0$ for all $\alpha\geq\alpha_\varepsilon$. This implies $\lvert x_\alpha\rvert\leq\varepsilon e$ for all 
$\alpha\geq\alpha_\varepsilon$. Hence $x_\alpha\xrightarrow{\lVert\cdot\rVert_e}0$.
\end{itemize}
\end{proof}

Now we are ready to prove our main result, whose proof is motivated by \cite{W74}.
\begin{thm}\label{main}
Let $(X,\tau)$ be a topological vector lattice. The following statements are equivalent.
\begin{itemize}
\item[(i)] $\dim X<\infty$.

\item[(ii)] $x_\alpha\xrightarrow{\tau}0$ iff $x_\alpha\xrightarrow{o}0$ for any net $x_\alpha$ in $X$.
\end{itemize}
\end{thm}

\begin{proof}
The implication (i) $\Longrightarrow$ (ii) is trivial.\\
(ii) $\Longrightarrow$ (i). It follows from Lemma \ref{lemma 1}, that $X$ has a strong unit $e$, and $(X=I_e,\lVert\cdot\rVert_e)$ is a normed lattice. 
For a net $x_\alpha$ in $X$, $x_\alpha\xrightarrow{\lVert\cdot\rVert_e}0$ $\Rightarrow$ $x_\alpha\xrightarrow{ru}0$ $\Rightarrow$ $x_\alpha\xrightarrow{o}0$ 
$\Rightarrow$ $x_\alpha\xrightarrow{\tau}0$. Combining this with Lemma \ref{lemma 1}(ii), we get $x_\alpha\xrightarrow{\lVert\cdot\rVert_e}0$ iff $x_\alpha\xrightarrow{\tau}0$.

Let $(\widehat{X},\lVert\cdot\rVert)$ be the norm completion of $(X,\lVert\cdot\rVert_e)$. Then $(\widehat{X},\lVert\cdot\rVert)$ is a Banach lattice. 
Let $\widehat{I_e}$ be the ideal generated by $e$ in $\widehat{X}$. Then it follows from \cite[Theorem 3.4]{AA02}, that $(\widehat{I_e},\widehat{\lVert\cdot\rVert}_e)$ 
is an AM-space with a strong unit $e$, where $\widehat{\lVert z\rVert}_e:=\inf\{\lambda>0:\lvert z\rvert\leq\lambda e\}$. 
Now \cite[Theorem 3.6]{AA02} implies that $\widehat{I_e}$ is lattice isometric to $C(K)$-space for some compact Hausdorff space $K$ 
such that the strong unit $e$ is identified with the constant function $\mathds 1$ on $K$. 
Clearly,  $X=I_e$ is a sublattice of $\widehat{I_e}$ and so, we can identify elements of $X$ with continuous functions on $K$.

Let $t_0\in K$ and $g=\chi_{t_{0}}$ be the characteristic function of $\{t_0\}$. Define
$$
  F:=\{f\in X: f\geq0 \ \text{and} \ f(t_0)=1\}.
$$
Then $F$ is directed downward under the pointwise ordering. For each $\alpha\in F$, let $f_\alpha=\alpha$. 
Then, Urysohn's extension lemma assures that $f_\alpha\downarrow g$ pointwise. If $g\not\in C(K)$ then $f_\alpha\downarrow 0$ in $C(K)$, and so 
$f_\alpha\downarrow 0$ in $X$. That is $f_\alpha\xrightarrow{o}0$ in $X$, and hence $f_\alpha\xrightarrow{\tau}0$ or $f_\alpha\xrightarrow{\lVert\cdot\rVert_e}0$, 
which is a contradiction since $\lVert f_\alpha\rVert_e\geq 1$. Thus, $g\in C(K)$, and so $\{t_0\}$ is open in $K$. So $K$ is discrete and hence finite. 
Therefore, $\dim X<\infty$.
\end{proof}

\begin{rem}
In \cite[p. 162]{VK73} an example is given of an infinite-dimensional ordered normed space $($which is not a normed lattice$)$, where the norm convergence coincides with the order convergence.
\end{rem}

In what follows, we show that, in atomic order continuous Banach lattices, the order convergence can be topologized on order intervals. 
The following result could be known, but since we do not have an appropriate reference, we include its proof for the sake of completeness.

\begin{prop}\label{atomipointwiseconvergence}
Let $X$ be an atomic vector lattice. Then a net $x_\alpha$ $uo$-converges iff it converges pointwise. 
\end{prop}

\begin{proof}
Without loss of generality, we can assume that the net $x_\alpha$ is in $X_+$ and converges to 0. The forward implication is obvious. 

For the converse, let $x_\alpha$ be a pointwise null in $X$.  Given $u\in X_+$, we need to show that $x_\alpha\wedge u\oc 0$. 
Let $\Delta=\mathcal{P}_{fin}(\Omega)\times\mathbb{N}$, where $\Omega$ is the collection of all atoms in $X$. 
The set $\Delta$ is directed w.r. to the following ordering: $(A,n)\le(B,m)$ if $A\subseteq B$ and $n\le m$. 
For each $\delta=(F,n)\in\Delta$, put $y_\delta=\frac{1}{n}\sum\limits_{a\in F}P_au+\sum\limits_{a\in \Omega\setminus F}P_au$, 
where $P_a$ denotes the band projection onto $span\{a\}$. It is easy to see that $y_\delta\downarrow 0$ and, for any $\delta\in\Delta$, 
there is $\alpha_\delta$ such that we have $0\leq x_\alpha\wedge u\leq y_\delta$ for any $\alpha\geq\alpha_\delta$. 
Therefore, $x_\alpha\wedge u\oc 0$.
\end{proof}

Unlike Theorem \ref{main}, the next theorem shows that $uo$-convergence is topological in any atomic vector lattice.    

\begin{thm}\label{uo-conv is top in discr VL}
The $uo$-convergence is topological in atomic vector lattices.
\end{thm} 

\begin{proof}
By Proposition \ref{atomipointwiseconvergence}, $uo$-convergence in atomic vector lattices is the same as pointwise convergence and therefore is topological. 
\end{proof}

Clearly, $o$-convergence is nothing than eventually order bounded $uo$-con\-ver\-gence. Replacing ``eventually order bounded''
by ``order bounded'', we obtain the following result in atomic vector lattices.

\begin{cor}\label{o-convergence is topological}
Let $X$ be an atomic vector lattice. Then order convergence is topological on every order bounded subset of $X$.
\end{cor}

\begin{proof}
By Theorem \ref{uo-conv is top in discr VL}, $uo$-convergence is topological in $X$ and hence on any subset of $X$ in the induced topology. 
Since order $o$-convergence coincides with $uo$-convergence on order intervals, we conclude that order convergence is also topological on order bounded subsets of $X$.
\end{proof}

The following result extends \cite[Theorem 5.3]{DOT}.

\begin{thm}\label{order cont. and atomic}
Let $X$ be a Banach lattice. The following statements are equivalent.
\begin{itemize}
\item[(i)] 	For any net $x_\alpha$ in $X$, $x_\alpha\xrightarrow{uo}0$ $\Longleftrightarrow$ $x_\alpha\xrightarrow{un}0$.
\item[(ii)] For any sequence $x_n$ in $X$, $x_n\xrightarrow{uo}0$ $\Longleftrightarrow$ $x_n\xrightarrow{un}0$.
\item[(iii)] $X$ is order continuous and atomic.
\end{itemize}
\end{thm}

\begin{proof}
(i) $\Longrightarrow$ (ii) is trivial. (ii) $\Longrightarrow$ (iii) is part of \cite[Theorem 5.3]{DOT}. For (iii) $\Longrightarrow$ (i), suppose $X$ is order continuous and atomic. 
Then, it follows from \cite[Corollary 4.14]{KMT}, that $x_\alpha\xrightarrow{un}0$ iff $P_ax_{\alpha}\rightarrow 0$ for any atom $a\in X$ and,
by Proposition \ref{atomipointwiseconvergence}, this holds iff $x_\alpha\xrightarrow{uo}0$.
\end{proof}

The following result is a partial converse of Corollary \ref{o-convergence is topological}.

\begin{thm}\label{converse}
Assume that there is a Hausdorff locally solid topology $\tau$ on an order continuous Banach lattice $X$ such that order convergence and $\tau$-convergence 
coincide on each order interval of $X$. Then $X$ is atomic. 
\end{thm}

\begin{proof}
First we show that $\tau$ is a Lebesgue topology. Assume $x_\alpha\xrightarrow{o}0$, then there exist $\alpha_0$ and $\nu\in X_+$ such that $(x_\alpha)_{\alpha\geq\alpha_0}\subseteq[-\nu,\nu]$. 
By the hypothesis, $(x_\alpha)_{\alpha\geq\alpha_0}\xrightarrow{\tau}0$ in $[-\nu,\nu]$, and so $x_\alpha \xrightarrow{\tau} 0$ in $X$.

Let $x_\alpha\xrightarrow{uo}0$. Since $X$ is order continuous, then $x_\alpha\xrightarrow{un}0$. 
Suppose now $x_\alpha\xrightarrow{un}0$, and take $u\in X_+$. Then $\big\lVert\lvert x_\alpha\rvert\wedge u\big\rVert\to 0$. 
Since the net $\lvert x_\alpha\rvert\wedge u$ is order bounded, then, by \cite[Theorem 4.22]{AB03}, $\lvert x_\alpha\rvert\wedge u\xrightarrow{\tau}0$ in $[-u,u]$, 
and so $\lvert x_\alpha\rvert\wedge u\xrightarrow{o}0$. We conclude $x_\alpha\xrightarrow{uo}0$. Thus $x_\alpha\xrightarrow{uo}0$ $\Longleftrightarrow$ $x_\alpha\xrightarrow{un}0$.
It follows from Theorem \ref{order cont. and atomic} that $X$ is atomic.
\end{proof}

\end{document}